\documentclass{article}
\usepackage{amsmath,amssymb,amsthm}
\usepackage{fullpage}

\newtheorem{theorem}{Theorem}[section]
\newtheorem{proposition}[theorem]{Proposition}
\newtheorem{lemma}[theorem]{Lemma}
\newtheorem{corollary}[theorem]{Corollary}
\theoremstyle{definition}
\newtheorem{definition}[theorem]{Definition}

\newcommand{\FF}{\mathbb{F}}
\newcommand{\ZZ}{\mathbb{Z}}

\title{Ordinary $3$-Isogeny Graphs and Improvement of Supersingularity Testing for Twisted Hessian Curves over Prime Fields}
\author{Yuji Hashimoto${}^{123}$ \qquad Koji Nuida${}^{23}$\medskip\\
${}^1$ School of Science and Engineering, Tokyo Denki University, Japan\\ (\texttt{y.hashimoto@mail.dendai.ac.jp})\\
${}^2$ Institute of Mathematics for Industry (IMI), Kyushu University, Japan\\ (\texttt{nuida@imi.kyushu-u.ac.jp})\\
${}^3$ Cyber Physical Security Research Institute (CPSEC),\\ National Institute of Advanced Industrial Science and Technology (AIST), Japan}
\date{\today}

\begin{document}

\maketitle

\begin{abstract}
For any primes $p \neq \ell$, $\ell$-isogeny graphs of ordinary elliptic curves defined over $\FF_{p^2}$ have a typical structure called $\ell$-volcanoes, and the structure is the core of Sutherland's supersingularity testing algorithm for elliptic curves.
In this paper, by exploiting the properties of $3$-isogenies between twisted Hessian curves, we show that when $p \equiv 2 \pmod{3}$ and $\ell = 3$, every ordinary twisted Hessian curve defined over $\FF_p$ lies on the surface of the $3$-volcano.
As an application, we give an improved version of Sutherland's supersingularity testing algorithm specialized to twisted Hessian curves defined over $\FF_p$ with $p \equiv 2 \pmod{3}$.
We also give a generalization of the known fact that any supersingular $j$-invariant is a cube in $\FF_{p^2}$; we show that for any twisted Hessian curve $H(a,d)$ defined over $\FF_{p^2}$, its $j$-invariant is not a cube in $\FF_{p^2}$ if and only if $H(a,d)$ is ordinary and lies on the floor of a $3$-volcano.
\end{abstract}

\section{Introduction}
\label{sec:introduction}

Let $p$ be a prime with $p \geq 5$.
Elliptic curves $E$ over $\overline{\FF_p}$ are classified into ordinary elliptic curves and supersingular elliptic curves, where $E$ is called supersingular if its rational point group has no element of order $p$.
Among the differences of properties between ordinary elliptic curves and supersingular ones, in this paper, we focus on the different structures of connected components of isogeny graphs consisting of ordinary elliptic curves and those consisting of supersingular ones.
Recall that for a power $q$ of $p$ and for a prime $\ell \neq p$, the $\ell$-isogeny graph $G_{\ell}(\FF_q)$ over $\FF_q$ consists of isomorphism classes of elliptic curves defined over $\FF_q$ as its vertices and $\ell$-isogenies between those elliptic curves as its edges.
Then it is known that connected components of $G_{\ell}(\FF_q)$ consisting of ordinary elliptic curves, which we call ordinary components, have typical structures called $\ell$-volcanoes; see Kohel's Ph.D.\ thesis \cite{Kohel:thesis} and papers by Fouquet--Morain \cite{DBLP:conf/ants/FouquetM02} and by Sutherland \cite{Sutherland:ANTSX}.
In contrast, it was shown by Pizer \cite{pizer1990ramanujan} that for $q = p^2$, the connected component of $G_{\ell}(\FF_{p^2})$ consisting of supersingular elliptic curves, which we call the supersingular component, is a Ramanujan graph.
The latter Ramanujan graph on supersingular elliptic curves is applied to cryptography such as the CGL hash function \cite{JC:ChaLauGor09}.
On the other hand, the former $\ell$-volcano graphs on ordinary elliptic curves are applied to supersingularity testing algorithms for elliptic curves proposed by Sutherland \cite{DBLP:journals/lmsjcm/Sutherland12a}.

In more details, the vertex set of each ordinary component of $G_{\ell}(\FF_q)$ is divided into some parts $V_0$, $V_1$, $\dots$, $V_h$, where $V_0$ and $V_h$ are called the surface and the floor of the $\ell$-volcano, respectively.
Moreover, $h$ is called the height of the $\ell$-volcano.
Now when $q = p^2$, the core idea of Sutherland's supersingularity testing algorithm is to compute a non-backtracking chain of $\ell$-isogenies $E_0 \to E_1 \to \cdots$ from an input curve $E_0$ in a way that if $E_0$ is ordinary, then each isogeny $E_j \to E_{j+1}$ is a descending isogeny, that is, from a curve in some $V_i$ to a curve in $V_{i+1}$.
Then after at most $h$ steps, the path arrives at the floor of the $\ell$-volcano, and a next non-backtracking $\ell$-isogeny goes to an elliptic curve outside of $G_{\ell}(\FF_{p^2})$, which tells us that $E_0$ is ordinary.
On the other hand, if such an $\ell$-isogeny is not found within $h + 1$ steps, then we can conclude that $E_0$ is supersingular.
In contrast to a supersingularity testing algorithm exploiting the order of a random point on the input curve (see \cite[Algorithm 1]{DBLP:journals/lmsjcm/Sutherland12a}), which is faster in average but has non-zero error probability, an advantage of Sutherland's algorithm is that its output is always correct.
The original algorithm of Sutherland \cite{DBLP:journals/lmsjcm/Sutherland12a} used $2$-isogenies between short Weierstrass curves, while Hashimoto and Nuida \cite{DBLP:conf/casc/HashimotoN21,DBLP:journals/ieiceta/HashimotoN23} gave an improved algorithm using $2^2$-isogenies between Legendre curves.

Now a problem in Sutherland's algorithm (and its improved version) is that it is in general computationally expensive to decide if an $\ell$-isogeny is a descending isogeny.
To circumvent this problem, Sutherland's algorithm computes a sufficient number (precisely, three) of paths $E_0 \to E_1 \to \cdots$ in parallel from the input curve $E_0$ to ensure that at least one of the starting isogenies $E_0 \to E_1$ is a descending isogeny.
In contrast, if we could select a descending isogeny from an input curve $E_0$ explicitly, then we would be able to improve the efficiency of Sutherland's algorithm by avoiding computation of multiple paths $E_0 \to E_1 \to \cdots$.
In this paper, we tackle this problem in some special case.

\subsection{Our Contribution}
\label{sec:introduction:contribution}

In this paper, we focus on ordinary components of $3$-isogeny graphs $G_3(\FF_{p^2})$ and discuss their application to improvement of Sutherland's supersingularity testing algorithm for elliptic curves.
First we show that an improved bound for the height $h$ of an ordinary component of $G_2(\FF_{p^2})$ as a $2$-volcano recently given by Hashimoto and Nuida in \cite{hashimoto2024boundsheights2isogenygraphs} can be generalized to the case of $\ell \neq 2$ (Proposition \ref{prop:bound_for_height}).
Then we study some properties of $3$-isogeny graphs on elliptic curves expressed by twisted Hessian forms (see e.g., a paper by Bernstein--Chuengsatiansup--Kohel--Lange \cite{LC:BCKL15}) $H(a,d) \colon a X^3 + Y^3 + Z^3 = d XYZ$ with $a,d \in \FF_{p^2}$, based on a work by Broon--Dang--Fouotsa--Moody \cite{perez2021isogenies} that describes $3$-isogenies from $H(a,d)$ explicitly.
For example, we prove (in Corollary \ref{cor:when_j_invariant_is_a_cube}) that the $j$-invariant of a twisted Hessian curve $H(a,d)$ with $a,d \in \FF_{p^2}$ is not a cube in $\FF_{p^2}$ if and only if $H(a,d)$ is ordinary and lies on the floor of the connected component of $G_3(\FF_{p^2})$ as a $3$-volcano.
This result, combined with a fact that any supersingular $j$-invariant can be realized by a twisted Hessian curve $H(a,d)$ with $a,d \in \FF_{p^2}$ (see Proposition \ref{prop:from_Weierstrass_to_Hesse}), gives another proof of a fact, originally proved by Morton \cite[Theorem 1.2]{morton2011cubic}, that any supersingular $j$-invariant is a cube in $\FF_{p^2}$.
Now we can give a variant of Sutherland's algorithm using $3$-isogenies based on the results of this paper, but our computer experiment showed that this variant does not outperform the aforementioned algorithm in \cite{DBLP:conf/casc/HashimotoN21,DBLP:journals/ieiceta/HashimotoN23} using $2^2$-isogenies between Legendre curves.

Then we focus on a special case that $p \equiv 2 \pmod{3}$ and the input for a supersingularity testing algorithm is a twisted Hessian curve $H(a,d)$ defined over $\FF_p$ (i.e., with $a,d \in \FF_p$).
In this case, we prove (in Theorem \ref{thm:surface_2_mod_3}) that any ordinary twisted Hessian curve $H(a,d)$ defined over $\FF_p$ lies on the surface $V_0$ of the connected component of $G_3(\FF_{p^2})$ as a $3$-volcano, and determine (in Corollary \ref{cor:surface_2_mod_3}) which $3$-isogenies from this $H(a,d)$ lie on the surface $V_0$ and which ones are descending isogenies.
As mentioned above, this result yields an improvement of Sutherland's supersingularity testing algorithm for the case of twisted Hessian curves $H(a,d)$ defined over $\FF_p$ by avoiding computation of multiple paths on the $3$-isogeny graph.
As summarized in Table \ref{tab:experiment_Fp_SS} of Section \ref{sec:SStest}, our computer experiment showed that in this special case with $p \equiv 2 \pmod{3}$ and $a,d \in \FF_p$, our proposed algorithm (Algorithm 2) reduced the computing times to $36.6\%$--$46.4\%$ of the existing algorithm of \cite{DBLP:conf/casc/HashimotoN21,DBLP:journals/ieiceta/HashimotoN23}.

\section{Preliminaries}
\label{sec:preliminaries}

In this paper, let $p$ be a prime with $p \geq 5$, and let $q$ be a power of $p$.
Let $\FF_q$ denote the finite field of order $q$, and let $\overline{\FF_p}$ denote an algebraic closure of $\FF_p$.
Let $\ell$ be a prime with $\ell \neq p$.

\subsection{Elliptic Curves}
\label{sec:preliminaries:EC}

It is known (as we have assumed that $p \geq 5$) that any elliptic curve $E$ defined over $\FF_q$, denoted shortly by $E/\FF_q$, can be expressed by a \emph{short Weierstrass form} $E \colon Y^2 Z = X^3 + a X Z^2 + b Z^3$ where $a,b \in \FF_q$ with $-16 (4 a^3 + 27 b^2) \neq 0$.
Let $O = O_E$ denote the identity element of the rational point group $E(\FF_q)$ of $E$, i.e., $O_E = (0:1:0)$.
Then the \emph{$j$-invariant} of $E$ is given by $j(E) = 1728 \cdot 4 a^3 / (4 a^3 + 27 b^2)$.
Two elliptic curves $E$ and $E'$ defined over $\overline{\FF_p}$ are isomorphic over $\overline{\FF_p}$ if and only if $j(E) = j(E')$.
We say that an elliptic curve $E/\overline{\FF_p}$ is \emph{supersingular} if $E[p] = \{O\}$, where we write $E[n] := \{ P \in E(\overline{\FF_p}) \mid [n] P = O \}$ for any integer $n \geq 1$, and \emph{ordinary} otherwise.
We also say that an element $j_0$ of $\overline{\FF_p}$ is \emph{ordinary} (resp.\ \emph{supersingular}) if $j(E) = j_0$ for some ordinary (resp.\ supersingular) elliptic curve $E/\overline{\FF_p}$.
It is known \cite[Theorem V.3.1]{Silverman:AEC} that if $j_0 \in \overline{\FF_p}$ is supersingular, then $j_0 \in \FF_{p^2}$.
Moreover, it is known \cite[Example V.4.4]{Silverman:AEC} that $0 \in \FF_{p^2}$ is ordinary if $p \equiv 1 \pmod{3}$ and is supersingular if $p \equiv 2 \pmod{3}$.

On the other hand, a \emph{twisted Hessian form} \cite{LC:BCKL15} of an elliptic curve over $\FF_q$ is given by
\[
H(a,d) \colon a X^3 + Y^3 + Z^3 = d X Y Z
\]
where $a,d \in \FF_q$ with $a (27a - d^3) \neq 0$.
Here the identity element of its rational point group is $O = (0:-1:1)$.
$H(a,d)$ is called a \emph{Hessian form} if $a = 1$.
Note that $H(a,d)$ is isomorphic to $H(1,d/c)$ where $c \in \overline{\FF_p}$ is a cubic root of $a$, via a change of coordinates $(X',Y',Z') = (cX,Y,Z)$.
Then the $j$-invariant of $H(a,d)$ is given \cite{farashahi2011numberdistinctlegendrejacobi} by
\begin{equation}
\label{eq:j-invariant_Hesse}
j(H(a,d))
= j(H(1,d/c))
= \left( (d/c) \frac{ (d/c)^3 + 216 }{ (d/c)^3 - 27 } \right)^3
= \frac{ d^3 }{ a } \frac{ ( d^3 + 216a )^3 }{ (d^3 - 27a)^3 } \enspace.
\end{equation}

When an element $j_0$ of $\FF_{p^2}$ is supersingular, a twisted Hessian curve $H(a,d)$ with $a,d \in \FF_{p^2}$ satisfying $j(H(a,d)) = j_0$ can be constructed by the following procedure (on the other hand, if this procedure failed for a given $j_0$, then we can conclude that $j_0$ is ordinary):
\begin{enumerate}
\item
Compute an elliptic curve $E_0 \colon Y^2 Z = X^3 + A_0 X Z^2 + B_0 Z^3$ in short Weierstrass form with $A_0, B_0 \in \FF_{p^2}$ satisfying $j(E_0) = j_0$; precisely, $A_0 := 0$ and $B_0 := 1$ if $j_0 = 0$; $A_0 := 1$ and $B_0 := 0$ if $j_0 = 1728$; and $A_0 := 3 j_0 / (1728 - j_0)$ and $B_0 := 2 j_0 / (1728 - j_0)$ if $j_0 \not\in \{0,1728\}$ \cite[Section 2.7]{Washington:book}.
Now by the assumption on $j_0$, $E_0$ is supersingular, therefore (as $p \geq 5$) we have $|E_0(\FF_{p^2})| = (p \pm 1)^2$ for some sign $\pm$ (\cite{ruck1987note}; see also \cite[Section 2.3]{DBLP:journals/lmsjcm/Sutherland12a}).
\item
Construct a quadratic twist $E_1 \colon Y^2 Z = X^3 + A_1 X Z^2 + B_1 Z^3$ of $E_0$, where $A_1 = \delta^2 A_0$ and $B_1 = \delta^3 B_0$ with $\delta \in (\FF_{p^2})^{\times}$ being a quadratic non-residue in $\FF_{p^2}$.
Then we have $j(E_1) = j(E_0)$ and $|E_0(\FF_{p^2})| + |E_1(\FF_{p^2})| = 2 p^2 + 2$ \cite[Exercise 4.10]{Washington:book}, therefore $|E_1(\FF_{p^2})| = (p \mp 1)^2$.
This implies that one of $|E_0(\FF_{p^2})|$ and $|E_1(\FF_{p^2})|$ is a multiple of $3$.
\item
Find a point $P = (x_0:y_0:1)$ of order $3$ in $E_h(\FF_{p^2})$ for some $h \in \{0,1\}$ by, e.g., finding a root $x_0 \in \FF_{p^2}$ of the third division polynomial $\psi_3 = 3 x^4 + 6A_h x^2 + 12B_h x - A_h{}^2$ and finding a square root $y_0 \in \FF_{p^2}$ of $x_0{}^3 + A_h x_0 + B_h$.
\item
By using an algorithm in \cite[Theorems 5.2 and 5.3]{LC:BCKL15} based on the point $P$ of order $3$ in $E_h(\FF_{p^2})$, construct a twisted Hessian curve $H(a,d)$ with $a,d \in \FF_{p^2}$ that is isomorphic to $E_h$, hence having $j$-invariant $j(H(a,d)) = j(E_h) = j_0$.
\end{enumerate}
We note also that when $p \equiv 2 \pmod{3}$ and $j_0 \in \FF_p$ is supersingular, as now $p \geq 5$, the elliptic curve $E_0$ in Step 1 satisfies that $|E_0(\FF_p)| = p + 1 \equiv 0 \pmod{3}$ \cite[Exercise V.5.10]{Silverman:AEC}.
Therefore, in this case, a point $P$ of order $3$ in Step 3 can be found in $E_0(\FF_p)$ (rather than in $E_0(\FF_{p^2})$), and consequently, the twisted Hessian curve $H(a,d)$ in Step 4 can be constructed to satisfy $a,d \in \FF_p$ (rather than $a,d \in \FF_{p^2}$).
We summarize the argument as a proposition for reference purpose.

\begin{proposition}
\label{prop:from_Weierstrass_to_Hesse}
Let $j_0 \in \FF_{p^2}$.
If $j_0$ is supersingular, then there exists a twisted Hessian curve $H(a,d)$ with $a,d \in \FF_{p^2}$ satisfying $j(H(a,d)) = j_0$, as constructed by the algorithm above.
If moreover $p \equiv 2 \pmod{3}$ and $j_0 \in \FF_p$, then the Hessian curve $H(a,d)$ can be found in a way that $a,d \in \FF_p$.
On the other hand, if the algorithm above failed for a given $j_0$, then $j_0$ is ordinary. 
\end{proposition}

We also note the following property of twisted Hessian curves define over $\FF_p$.

\begin{proposition}
\label{prop:ordinary_if_1_mod_3}
Assume that $p \equiv 1 \pmod{3}$.
Then any twisted Hessian curve $H(a,d)$ with $a,d \in \FF_p$ is ordinary.
\end{proposition}
\begin{proof}
As $p \equiv 1 \pmod{3}$, $\FF_p$ involves a primitive third root $\omega$ of unity.
Then $(0:-\omega:1)$ is a point of order $3$ in $H(a,d)$ \cite[Theorem 5.1]{LC:BCKL15}, therefore $|H(a,d)(\FF_p)| \equiv 0 \pmod{3}$.
On the other hand, if $H(a,d)$ were supersingular, then (as $H(a,d)$ is defined over $\FF_p$ and $p \geq 5$) we would have $|H(a,d)(\FF_p)| = p + 1 \not\equiv 0 \pmod{3}$ by \cite[Exercise V.5.10]{Silverman:AEC}.
Hence $H(a,d)$ is ordinary, as desired.
\end{proof}

\subsection{Isogenies}
\label{sec:preliminaries:isogeny}

Let $E$ and $E'$ be elliptic curves defined over $\FF_q$.
An \emph{isogeny} $\phi \colon E \to E'$ over $\FF_q$ is a non-constant $\FF_q$-morphism with $\phi(O_E) = O_{E'}$; then $\phi$ also defines a group homomorphism $E(\FF_q) \to E'(\FF_q)$.
We say that $E$ and $E'$ are \emph{isogenous} if there is an isogeny from $E$ to $E'$.
It is known that if an elliptic curve $E'$ is isogenous to an ordinary (resp.\ a supersingular) elliptic curve $E$, then $E'$ is also ordinary (resp.\ supersingular).
By this and the fact (mentioned in Section \ref{sec:preliminaries:EC}) that all supersingular $j$-invariants lie in $\FF_{p^2}$, it follows that if there is a chain of isogenies $E_0 \to E_1 \to \cdots \to E_n$ and $j(E_n) \not\in \FF_{p^2}$, then $E_0$ is ordinary.
This is one of the key ingredients in Sutherland's supersingularity testing algorithm for elliptic curves \cite{DBLP:journals/lmsjcm/Sutherland12a}.

We say that an isogeny $\phi \colon E \to E'$ is \emph{separable} if the field extension $\FF_q(E) / \phi^*(\FF_q(E'))$ of the function fields induced by $\phi$ is separable.
Then an isogeny $\phi \colon E \to E'$ is called an \emph{$\ell$-isogeny} if $\phi$ is separable and $|\ker \phi| = \ell$.
We say that $E$ and $E'$ are \emph{$\ell$-isogenous} if there is an $\ell$-isogeny from $E$ to $E'$.

Explicit forms of $3$-isogenies from a twisted Hessian curve $H(a,d)$ are described in \cite{perez2021isogenies} (we note again that now $p \geq 5$).

\begin{proposition}
[{\cite[Theorem 3]{perez2021isogenies}}]
\label{prop:3-isogeny_Hesse_type1}
Suppose that $c \in \FF_q$ and $c^3 = a$.
Then the map
\[
\phi \colon (X:Y:Z) \mapsto ( XYZ : c^2 X^2 Z + c X Y^2 + Y Z^2 : c^2 X^2 Y + c X Z^2 + Y^2 Z )
\]
is a $3$-isogeny from $H(a,d)$ to $H(A,D)$, where $A := d^2 c + 3dc^2 + 9a$ and $D := d + 6c$, with
\[
\ker \phi = \{ (0:-1:1), (1:-c:0), (1:0:-c) \} \enspace.
\]
We call this $\phi$ a $3$-isogeny of \emph{type 1} associated with $c$.
\end{proposition}

\begin{proposition}
[{\cite[Theorem 1]{perez2021isogenies}}]
\label{prop:3-isogeny_Hesse_type2}
Suppose that $\omega \in \FF_q$, $\omega^3 = 1$, and $\omega \neq 1$.
Then the map
\[
\phi \colon (X:Y:Z) \mapsto ( XYZ : aX^3 + \omega^2 Y^3 + \omega Z^3 : aX^3 + \omega Y^3 + \omega^2 Z^3)
\]
is a $3$-isogeny from $H(a,d)$ to $H(d^3 - 27a, 3d)$ with
\[
\ker \phi = \{ (0:-1:1), (0:-\omega:1), (0:-\omega^2:1) \} \enspace.
\]
We call this $\phi$ a $3$-isogeny of \emph{type 2}.
\end{proposition}

\subsection{Isogeny Graphs}
\label{sec:preliminaries:isogeny_graphs}

We define the \emph{$\ell$-isogeny graph} $G_{\ell}(\FF_{p^2})$ over $\FF_{p^2}$ to be the graph with vertex set $\FF_{p^2}$ whose edge $j_1 \to j_2$ corresponds to an $\ell$-isogeny from an elliptic curve $E_1/\FF_{p^2}$ with $j(E_1) = j_1$ to an elliptic curve $E_2/\FF_{p^2}$ with $j(E_2) = j_2$ (where $\ell$-isogenies with the same kernel are regarded as the same isogeny).
As described in \cite{Sutherland:ANTSX}, when $j_1,j_2 \not\in \{0,1728\}$, edges from $j_1$ to $j_2$ have the same multiplicity as edges from $j_2$ to $j_1$; therefore such edges can be regarded as undirected edges by making a pair of each edge $j_1 \to j_2$ with an associated edge $j_2 \to j_1$.
Then each connected component of $G_{\ell}(\FF_{p^2})$ consists either of ordinary $j$-invariants only or of supersingular $j$-invariants only; we call a connected component of the former and the latter types an \emph{ordinary} component and a \emph{supersingular} component, respectively.

To describe the structure of ordinary components of $G_{\ell}(\FF_{p^2})$, we recall the following definition of $\ell$-volcanoes (see e.g., \cite{Sutherland:ANTSX}) whose name originates from \cite{DBLP:conf/ants/FouquetM02}.

\begin{definition}
[$\ell$-volcano]
\label{defn:volcano}
An \emph{$\ell$-volcano} is a finite connected undirected graph whose vertex set is partitioned into subsets $V_0$, $V_1$, $\dots$, $V_h$ satisfying the following properties:
\begin{itemize}
\item
The subgraph on $V_0$, called the \emph{surface}, is a regular graph of degree at most $2$ (possibly with self-loops or parallel edges).
\item
For each $0 \leq i \leq h-1$, each vertex in $V_i$ has degree $\ell + 1$.
\item
For each $1 \leq i \leq h$, each vertex in $V_i$ has exactly one edge to a vertex in $V_{i-1}$; such edges are the only edges in the graph not on the surface.
\end{itemize}
We call the value $h$ the \emph{height} of the $\ell$-volcano, and we call $V_h$ the \emph{floor} of the $\ell$-volcano.
\end{definition}

Then we have the following property.

\begin{proposition}
[{Kohel \cite{Kohel:thesis}; Sutherland \cite[Theorem 7]{Sutherland:ANTSX}}]
\label{prop:volcano_graph}
Let $V$ be an ordinary component of $G_{\ell}(\FF_{p^2})$ that does not contain $0$ nor $1728$.
Then $V$ is an $\ell$-volcano with height $h = \nu_{\ell}((t^2 - 4p^2) / D_0) / 2$, where $t^2 = (\mathrm{Tr} \pi_E)^2$ for any elliptic curve $E/\FF_{p^2}$ with $j(E) \in V$, $\mathrm{Tr} \pi_E$ is the trace of the $p^2$-power Frobenius endomorphism $\pi_E$ on $E$, $D_0 = \mathrm{disc}(\mathcal{O}_0)$, $\mathcal{O}_0$ is the endomorphism ring of any elliptic curve $E/\FF_{p^2}$ with $j(E) \in V_0$, and $\nu_{\ell}$ denotes the $\ell$-adic valuation.
\end{proposition}

\begin{proposition}
[{Sutherland \cite[Remark 8]{Sutherland:ANTSX}}]
\label{prop:volcano_graph_0_1728}
Let $V$ be an ordinary component of $G_{\ell}(\FF_{p^2})$ that contains $0$ or $1728$.
Then $V$ is of one of the following types:
\begin{itemize}
\item
$V$ is (when viewed as an undirected graph) an $\ell$-volcano with height $h = 0$ as in Proposition \ref{prop:volcano_graph}.
\item
$V$ satisfies the conditions for an $\ell$-volcano with height $h$ as in Proposition \ref{prop:volcano_graph} except that the surface $V_0$ consists of a single vertex $0$; $V_1$ consists of $(\ell - (\frac{-3}{\ell})) / 3$ vertices; and each vertex in $V_1$ has three incoming edges from $0$ but only one outgoing edge to $0$.
\item
$V$ satisfies the conditions for an $\ell$-volcano with height $h$ as in Proposition \ref{prop:volcano_graph} except that the surface $V_0$ consists of a single vertex $1728$; $V_1$ consists of $(\ell - (\frac{-1}{\ell})) / 2$ vertices; and each vertex in $V_1$ has two incoming edges from $1728$ but only one outgoing edge to $1728$.
\end{itemize}
\end{proposition}

Let $V$ be an ordinary component of $G_{\ell}(\FF_{p^2})$ as in Proposition \ref{prop:volcano_graph} or Proposition \ref{prop:volcano_graph_0_1728}.
Then for any $\ell$-isogeny $\phi \colon E \to E'$ with $j(E),j(E') \in V$, we say that $\phi$ is an \emph{ascending isogeny} if $j(E) \in V_i$ and $j(E') \in V_{i-1}$ for some index $i$, and $\phi$ is a \emph{descending isogeny} if $j(E) \in V_i$ and $j(E') \in V_{i+1}$ for some index $i$.

In \cite[Equation (2)]{DBLP:journals/lmsjcm/Sutherland12a}, an upper bound $h < \log_{\ell} \sqrt{4p^2}$ for the height $h$ of an ordinary component of $G_{\ell}(\FF_{p^2})$ is given.
Here we slightly improve the bound for $h$; we note that a similar bound for the case $\ell = 2$ was given in \cite{hashimoto2024boundsheights2isogenygraphs}.

\begin{proposition}
\label{prop:bound_for_height}
Assume that $\ell \geq 3$.
Then the height $h$ of any ordinary component of $G_{\ell}(\FF_{p^2})$ (as an $\ell$-volcano) satisfies $h \leq \lfloor \log_{\ell}(4p) / 2 \rfloor$.
\end{proposition}
\begin{proof}
By the formula $h = \nu_{\ell}((t^2 - 4p^2) / D_0) / 2$ of the height $h$ in Proposition \ref{prop:volcano_graph}, we have $h \leq \nu_{\ell}(t^2 - 4p^2) / 2$.
Now we have $|t| \leq 2 \sqrt{p^2} = 2p$ by the Hasse bound, while we have $t \not\equiv 0 \pmod{p}$ as we now consider an ordinary component (see e.g., \cite[Proof of Theorem V.4.1]{Silverman:AEC}), therefore we have $|t| < 2p$.
Now we have
\[
\nu_{\ell}(t^2 - 4p^2)
= \nu_{\ell}((t + 2p)(t - 2p))
= \nu_{\ell}(t + 2p) + \nu_{\ell}(t - 2p) \enspace.
\]
On the other hand, we have $(t + 2p) - (t - 2p) = 4p$, while $\nu_{\ell}(4p) = 0$ as $\ell \neq 2$ and $\ell \neq p$, therefore we have $\nu_{\ell}(t + 2p) = 0$ or $\nu_{\ell}(t - 2p) = 0$.
Hence we have $\nu_{\ell}(t^2 - 4p^2) = \nu_{\ell}(t \pm 2p)$ for some sign $\pm$, therefore
\[
\nu_{\ell}(t^2 - 4p^2)
= \nu_{\ell}(t \pm 2p)
\leq \log_{\ell}(|t \pm 2p|)
\leq \log_{\ell}(|t| + 2p)
< \log_{\ell}(2p + 2p)
= \log_{\ell}(4p) \enspace.
\]
This implies that $h \leq \nu_{\ell}(t^2 - 4p^2) / 2 < \log_{\ell}(4p) / 2$, therefore we have $h \leq \lfloor \log_{\ell}(4p) / 2 \rfloor$, as desired.
\end{proof}

\section{$3$-Isogeny Walks with Twisted Hessian Curves}
\label{sec:3-isogeny_walk}

In this section, we study some properties of $3$-isogenies between twisted Hessian Curves.

\subsection{General Case}
\label{sec:3-isogeny_walk:general}

First, the following property is used for determining the backtracking edge in a walk on the $3$-isogeny graph.

\begin{proposition}
\label{prop:backtracking}
Let $\phi \colon H(a,d) \to H(A,D)$ be either one of the three $3$-isogenies of type $1$ from a twisted Hessian curve $H(a,d)$ (as in Proposition \ref{prop:3-isogeny_Hesse_type1}) or the $3$-isogeny of type $2$ from $H(a,d)$ (as in Proposition \ref{prop:3-isogeny_Hesse_type2}).
Then the $3$-isogeny $\psi \colon H(A,D) \to H(a',d')$ of type $2$ from the twisted Hessian curve $H(A,D)$ satisfies $H(a',d') \simeq H(a,d)$.
\end{proposition}
\begin{proof}
First we consider the case that $\phi$ is of type $1$.
By the shape of $\phi$ as in Proposition \ref{prop:3-isogeny_Hesse_type1}, the image of the point $(0:-\omega:1)$ of $H(a,d)$ by $\phi$, where $\omega^3 = 1$ and $\omega \neq 1$, is
\[
\phi(0:-\omega:1)
= ( 0 : -\omega \cdot 1^2 : (-\omega)^2 \cdot 1 )
= ( 0 : -\omega : \omega^2 )
= ( 0 : -\omega^2 : \omega^3 )
= ( 0 : -\omega^2 : 1 ) \enspace,
\]
therefore $\phi(0:-\omega:1) \in \ker \psi$.
Hence the composite isogeny $\psi \circ \phi \colon H(a,d) \to H(a',d')$ annihilates both $(1:-c:0)$ and $(0:-\omega:1)$, which generate the $3$-torsion subgroup $H(a,d)[3] \simeq \ZZ/3\ZZ \oplus \ZZ/3\ZZ$, therefore we have $\ker (\psi \circ \phi) = H(a,d)[3]$.
Hence $\psi \circ \phi$ is equal to the multiplication-by-$3$ map $[3]$ on $H(a,d)$ up to isomorphism, therefore $H(a',d') \simeq H(a,d)$, as desired.

Similarly, in the case that $\phi$ is of type $2$, by the shape of $\phi$ as in Proposition \ref{prop:3-isogeny_Hesse_type2}, the image of the point $(1:-c:0)$ of $H(a,d)$ by $\phi$, where $c^3 = a$, is
\[
\begin{split}
\phi(1:-c:0)
= ( 0 : a + \omega^2 (-c)^3 : a + \omega (-c)^3 )
&= ( 0 : a - \omega^2 a : a - \omega a ) \\
&= ( 0 : 1 - \omega^2 : 1 - \omega ) \\
&= ( 0 : 1 + \omega : 1 )
= ( 0 : - \omega^2 : 1 ) \enspace,
\end{split}
\]
therefore $\phi(1:-c:0) \in \ker \psi$.
Hence, by the same argument as the first case above, we have $\ker (\psi \circ \phi) = H(a,d)[3]$ and $\psi \circ \phi$ is equal to $[3]$ up to isomorphism, therefore $H(a',d') \simeq H(a,d)$, as desired.
This completes the proof.
\end{proof}

Secondly, the following property shows that if the domain curve $H(a,d)$ is defined over $\FF_{p^2}$ (i.e., $a,d \in \FF_{p^2}$) and $d \neq 0$, then whether the codomain curve of a $3$-isogeny of type $1$ associated with $c$ (where $c^3 = a$) has $j$-invariant lying in $\FF_{p^2}$ can be determined by just checking if $c \in \FF_{p^2}$.
We note that when $H(a,d)$ is supposed to be supersingular, the statement of Theorem \ref{thm:whether_codomain_is_over_Fp2} below follows from the known fact \cite[Theorem 1.2]{morton2011cubic} that any supersingular $j$-invariant is a cube in $\FF_{p^2}$.
In contrast, here we do not suppose that $H(a,d)$ is supersingular in the statement of Theorem \ref{thm:whether_codomain_is_over_Fp2}.

\begin{theorem}
\label{thm:whether_codomain_is_over_Fp2}
Let $H(a,d)$ be a twisted Hessian curve with $a,d \in \FF_{p^2}$ and $d \neq 0$.
Let $\phi \colon H(a,d) \to H(A,D)$ be a $3$-isogeny of type $1$ associated with $c \in \overline{\FF_p}$, where $c^3 = a$, $A = d^2 c + 3 d c^2 + 9a$, and $D = d + 6c$.
If $j(H(A,D)) \in \FF_{p^2}$, then $c \in \FF_{p^2}$.
\end{theorem}
\begin{proof}
Assume, for the sake of contradiction, that $c \not\in \FF_{p^2}$.
Let $\omega \in \FF_{p^2}$ be an element with $\omega^2 + \omega + 1 = 0$, which is a primitive third root of unity.
Then $a$ has three cubic roots $c$, $\omega c$, and $\omega^2 c$, none of which lies in $\FF_{p^2}$ (as $c \not\in \FF_{p^2}$ and $\omega \in \FF_{p^2}$).
Therefore $X^3 - a$ is the minimal polynomial of $c$ over $\FF_{p^2}$.
Now there is an automorphism $\sigma$ of the extension field $\FF_{p^2}(c)/\FF_{p^2}$ satisfying $\sigma(c) = \omega c$.

Let $d_0 := d/3$.
Then we have
\[
\begin{split}
j(H(A,D))
= \frac{ D^3 }{ A } \frac{ (D^3 + 216A)^3 }{ (D^3 - 27A)^3 }
&= \frac{ 3^3 (d_0 + 2c)^3 }{ 9 (d_0{}^2 c + d_0 c^2 + c^3) } \frac{ ( 3^3 (d_0 + 2c)^3 + 216 A )^3 }{ ( 3^3 (d_0 + 2c)^3 - 27 A )^3 } \\
&= \frac{ 3 (d_0 + 2c)^3 }{ c (d_0{}^2 + d_0 c + c^2) } \frac{ ( (d_0 + 2c)^3 + 8 A )^3 }{ ( (d_0 + 2c)^3 - A )^3 } \enspace.
\end{split}
\]
As $d_0{}^2 + d_0 c + c^2 = (d_0{}^3 - c^3) / (d_0 - c) = (d_0{}^3 - a) / (d_0 - c)$ (note that $d_0 - c \neq 0$, as $d_0 \in \FF_{p^2}$ and $c \not\in \FF_{p^2}$) and
\[
\begin{split}
(d_0 + 2c)^3 - A
&= (d_0{}^3 + 6 d_0{}^2 c + 12 d_0 c^2 + 8 c^3) - 9 (d_0{}^2 c + d_0 c^2 + c^3) \\
&= d_0{}^3 - 3 d_0{}^2 c + 3 d_0 c^2 - c^3
= (d_0 - c)^3 \enspace,
\end{split}
\]
we have
\[
\begin{split}
j(H(A,D))
= \frac{ 3 (d_0 + 2c)^3 }{ c (d_0{}^3 - a) / (d_0 - c) } \frac{ ( (d_0 + 2c)^3 + 8 A )^3 }{ (d_0 - c)^9 }
= \frac{ 3 (d_0 + 2c)^3 }{ c (d_0{}^3 - a) } \frac{ ( (d_0 + 2c)^3 + 8 A )^3 }{ (d_0 - c)^8 }
\in \FF_{p^2} \enspace,
\end{split}
\]
therefore, as $d_0{}^3 - a \in \FF_{p^2}$, we have
\[
\frac{ (d_0 + 2c)^3 }{ c } \frac{ ( (d_0 + 2c)^3 + 8 A )^3 }{ (d_0 - c)^8 }
\in \FF_{p^2} \enspace.
\]
This element is fixed by the automorphism $\sigma$, therefore we have
\[
\frac{ (d_0 + 2c)^3 }{ c } \frac{ ( (d_0 + 2c)^3 + 8 A )^3 }{ (d_0 - c)^8 }
= \frac{ (d_0 + 2 \omega c)^3 }{ \omega c } \frac{ ( (d_0 + 2 \omega c)^3 + 8 \sigma(A) )^3 }{ (d_0 - \omega c)^8 } \enspace.
\]
This implies that
\[
F
:= (d_0 + 2c)^3 ( (d_0 + 2c)^3 + 8 A )^3 \omega (d_0 - \omega c)^8
- (d_0 + 2 \omega c)^3 ( (d_0 + 2 \omega c)^3 + 8 \sigma(A) )^3 (d_0 - c)^8
= 0 \enspace.
\]

Now by expanding $F$ and using the relations $c^3 = a$ and $\omega^2 + \omega + 1 = 0$, it follows from a direct calculation that
\begin{equation}
\label{eq:thm:whether_codomain_is_over_Fp2:1}
\begin{split}
F
&= (F_2 \cdot 2^4 \cdot 3 \cdot 5 \cdot 7 \cdot 13 \cdot (2 \omega + 1) \cdot d_0{}^3 \cdot a^5) \cdot c^2 \\
&\quad + (F_1 \cdot 2 \cdot (\omega + 2) \cdot d_0 \cdot a^6) \cdot c \\
&\quad + (F_0 \cdot (\omega - 1) \cdot d_0{}^2 \cdot a^6) \\
&= 0
\end{split}
\end{equation}
where, by putting $z := d_0{}^3 / a$,
\[
F_2 := z^5 + 2614 z^4 - 32336 z^3 - 433360 z^2 - 57856 z + 41984 \enspace,
\]
\[
F_1 := 4 z^6 + 2790621 z^5 + 427259532 z^4 - 546940856 z^3 - 4822215552 z^2 - 316987392 z + 25907200 \enspace,
\]
\[
F_0 := z^6 + 621168 z^5 - 126479844 z^4 - 4334764688 z^3 - 184901856 z^2 + 4578771456 z + 66734080 \enspace.
\]
Indeed, by putting
\[
\widehat{F}_2
:= F_2 \cdot a^5
= d_0{}^{15} + 2614 d_0{}^{12} a - 32336 d_0{}^9 a^2 - 433360 d_0{}^6 a^3 - 57856 d_0{}^3 a^4 + 41984 a^5 \enspace,
\]
\[
\begin{split}
\widehat{F}_1
:= F_1 \cdot a^6
&= 4 d_0{}^{18} + 2790621 d_0{}^{15} a + 427259532 d_0{}^{12} a^2 - 546940856 d_0{}^9 a^3 \\
&\quad - 4822215552 d_0{}^6 a^4 - 316987392 d_0{}^3 a^5 + 25907200 a^6 \enspace,
\end{split}
\]
\[
\begin{split}
\widehat{F}_0
:= F_0 \cdot a^6
&= d_0{}^{18} + 621168 d_0{}^{15} a - 126479844 d_0{}^{12} a^2 - 4334764688 d_0{}^9 a^3 \\
&\quad - 184901856 d_0{}^6 a^4 + 4578771456 d_0{}^3 a^5 + 66734080 a^6 \enspace,
\end{split}
\]
we have
\begin{equation}
\label{eq:thm:whether_codomain_is_over_Fp2:2}
\begin{split}
F
= (\widehat{F}_2 \cdot 2^4 \cdot 3 \cdot 5 \cdot 7 \cdot 13 \cdot (2 \omega + 1) \cdot d_0{}^3) \cdot c^2 
+ (\widehat{F}_1 \cdot 2 \cdot (\omega + 2) \cdot d_0) \cdot c 
+ (\widehat{F}_0 \cdot (\omega - 1) \cdot d_0{}^2)
\end{split}
\end{equation}
in the quotient ring $\ZZ[a,d_0,\omega,c] / (c^3 - a, \omega^2 + \omega + 1)$ when $a$, $d_0$, $\omega$, and $c$ are temporarily regarded as indeterminates, therefore Eq.\eqref{eq:thm:whether_codomain_is_over_Fp2:2} still holds when taking reduction modulo $p$ and substituting the actual values of $a$, $d_0$, $\omega$, and $c$.
Hence Eq.\eqref{eq:thm:whether_codomain_is_over_Fp2:1} holds in $\overline{\FF_p}$.

Regarding Eq.\eqref{eq:thm:whether_codomain_is_over_Fp2:1}, as $F_0,F_1,F_2,d_0,a,\omega \in \FF_{p^2}$ and $X^3 - a$ is the minimal polynomial of $c$ over $\FF_{p^2}$, it follows that
\[
\begin{cases}
F_2 \cdot 2^4 \cdot 3 \cdot 5 \cdot 7 \cdot 13 \cdot (2 \omega + 1) \cdot d_0{}^3 \cdot a^5 = 0 \enspace,\\
F_1 \cdot 2 \cdot (\omega + 2) \cdot d_0 \cdot a^6 = 0 \enspace,\\
F_0 \cdot (\omega - 1) \cdot d_0{}^2 \cdot a^6 = 0 \enspace.
\end{cases}
\]
Now if $\omega + 2 = 0$ in $\FF_{p^2}$, then we have $\omega = -2$ and $1 = \omega^3 = -8$, therefore $9 = 0$, contradicting the fact $p \neq 3$.
Hence we have $\omega + 2 \neq 0$.
Similarly, if $2 \omega + 1 = 0$ in $\FF_{p^2}$, then we have $\omega = -1/2$ and $1 = \omega^3 = -1/8$, therefore $9 = 0$, contradicting the fact $p \neq 3$.
Hence we have $2 \omega + 1 \neq 0$.
By this and the facts $\omega \neq 1$, $p \not\in \{2,3\}$, $d_0 \neq 0$, and $a \neq 0$, it follows that
\[
F_2 \cdot 5 \cdot 7 \cdot 13 = 0 \,,\,
F_1 = 0 \,,\,
F_0 = 0 \enspace.
\]
We divide the proof into the following four cases.

\paragraph{Case 1: $p \not\in \{5,7,13\}$.}

In this case, we have $F_2 = F_1 = F_0 = 0$.
Now $F_2$ can be factored as
\[
F_2 = (z + 8) \cdot (z^2 - 20 z - 8) \cdot (z^2 + 2626 z - 656) \enspace,
\]
therefore we have $z + 8 = 0$, $z^2 - 20 z - 8 = 0$, or $z^2 + 2626 z - 656 = 0$.
We divide the proof into the following three cases.

\paragraph{Case 1-1: $z + 8 = 0$.}

Now the remainder of the polynomial $F_0$ modulo $z + 8$ is $1632586752000 = 2^{13} \cdot 3^{13} \cdot 5^3$, which must be zero as $F_0 = 0$ and $z + 8 = 0$.
This contradicts the assumption $p \not\in \{2,3,5\}$.

\paragraph{Case 1-2: $z^2 - 20 z - 8 = 0$.}

Now the remainders of the polynomials $F_0$ and $F_1$ modulo $z^2 - 20 z - 8$ are
\[
-2714622253056 z - 1066458742272
= -2^9 \cdot 3^8 \cdot 7^2 \cdot 19 \cdot 31 \cdot (28 z + 11)
\]
and
\[
3708367899264 z + 1454261921280
= 2^7 \cdot 3^9 \cdot 7^2 \cdot 19 \cdot 31 \cdot (51 z + 20) \enspace,
\]
respectively.
These two polynomials must be zero as $F_0 = F_1 = 0$ and $z^2 - 20 z - 8 = 0$.
We divide the proof into the following two cases.

\paragraph{Case 1-2-1: $p \not\in \{19,31\}$.}

In this case, by the assumption $p \not\in \{2,3,7\}$, we have $28z + 11 = 51z + 20 = 0$.
Hence
\[
0 = 51 \cdot (28 z + 11) - 28 \cdot (51 z + 20)
= 1 \enspace,
\]
a contradiction.

\paragraph{Case 1-2-2: $p \in \{19,31\}$.}

In this case, for any choice of $p$, we have $z^{(p^2 - 1) / 3} \bmod (z^2 - 20z - 8) = 1$ as polynomials in $z$ over $\FF_p$, therefore the element $z \in \FF_{p^2}$ satisfying $z^2 - 20z - 8 = 0$ also satisfies $z^{(p^2 - 1) / 3} = 1$.
Hence $z = d_0{}^3 / a$ is a cube in $\FF_{p^2}$.
This implies that $a$ is also a cube in $\FF_{p^2}$, contradicting the assumption that $c \not\in \FF_{p^2}$.

\paragraph{}

Hence we have a contradiction in all subcases for Case 1-2.

\paragraph{Case 1-3: $z^2 + 2626 z - 656 = 0$.}

Now the remainders of the polynomials $F_0$ and $F_1$ modulo $z^2 + 2626 z - 656$ are
\[
31682886407915875200 z - 7913936287671782400
= 2^7 \cdot 3^8 \cdot 5^2 \cdot 37 \cdot 53 \cdot (769532791 z - 192218392)
\]
and
\[
124498590050861874000 z - 31097984477624784000
= 2^4 \cdot 3^8 \cdot 5^3 \cdot 37 \cdot 53 \cdot (4838233297 z - 1208522152) \enspace,
\]
respectively.
These two polynomials must be zero as $F_0 = F_1 = 0$ and $z^2 + 2626 z - 656 = 0$.
We divide the proof into the following two cases.

\paragraph{Case 1-3-1: $p \not\in \{11,17,23,29,37,53\}$.}

In this case, by the assumption $p \not\in \{2,3,5\}$, we have
\[
769532791 z - 192218392
= 4838233297 z - 1208522152
= 0 \enspace.
\]
Hence
\[
\begin{split}
0 &= 4838233297 \cdot (769532791 z - 192218392) - 769532791 \cdot (4838233297 z - 1208522152) \\
&= 143687808
= 2^7 \cdot 3^2 \cdot 11 \cdot 17 \cdot 23 \cdot 29 \enspace,
\end{split}
\]
contradicting the assumption $p \not\in \{2,3,11,17,23,29\}$.

\paragraph{Case 1-3-2: $p \in \{11,17,23,29,37,53\}$.}

In this case, similarly to Case 1-2-2, for any choice of $p$, we have $z^{(p^2 - 1) / 3} \bmod (z^2 + 2626 z - 656) = 1$ as polynomials in $z$ over $\FF_p$, therefore the element $z \in \FF_{p^2}$ satisfying $z^2 + 2626 z - 656 = 0$ also satisfies $z^{(p^2 - 1) / 3} = 1$.
This implies a contradiction in a way similar to Case 1-2-2.

\paragraph{}

Hence we have a contradiction in all subcases for Case 1-3, therefore we have a contradiction in all subcases for Case 1.

\paragraph{Case 2: $p = 5$.}

In this case, we have
\[
\gcd(F_0,F_1)
= z^5 + 2 z^4 + 4 z^3 + 3 z^2 + z
= z (z + 3)^4
\]
as polynomials in $z$ over $\FF_p$, therefore $z (z + 3)^4 = 0$ as $F_1 = F_0 = 0$.
This implies that $z = -3$ (as $z = d_0{}^3 / a \neq 0$), which is a cube in $\FF_p$ (as now $p \equiv 2 \pmod{3}$).
This implies a contradiction in a way similar to Case 1-2-2.

\paragraph{Case 3: $p = 7$.}

In this case, we have
\[
\gcd(F_0,F_1)
= z^4 + 2 z^3 + 6 z^2 + 5 z + 1
= (z^2 + z + 6)^2
\]
as polynomials in $z$ over $\FF_p$, therefore $(z^2 + z + 6)^2 = 0$ as $F_1 = F_0 = 0$, hence $z^2 + z + 6 = 0$.
Now we have $z^{(p^2 - 1) / 3} \bmod (z^2 + z + 6) = 1$ as polynomials in $z$ over $\FF_p$.
This implies a contradiction in a way similar to Case 1-2-2.

\paragraph{Case 4: $p = 13$.}

In this case, we have
\[
\gcd(F_0,F_1)
= z^4 + 6 z^3 + 12 z^2 + 3 z + 5
= (z^2 + z + 12) (z^2 + 5 z + 8)
\]
as polynomials in $z$ over $\FF_p$, therefore $z^2 + z + 12 = 0$ or $z^2 + 5 z + 8 = 0$ as $F_1 = F_0 = 0$.
Now we have $z^{(p^2 - 1) / 3} \bmod (z^2 + z + 12) = 1$ and $z^{(p^2 - 1) / 3} \bmod (z^2 + 5 z + 8) = 1$ as polynomials in $z$ over $\FF_p$.
This implies a contradiction in a way similar to Case 1-2-2.

\paragraph{}

Hence we have a contradiction in any case.
Therefore we have $c \in \FF_{p^2}$, as desired.
This concludes the proof of Theorem \ref{thm:whether_codomain_is_over_Fp2}.
\end{proof}

Now we have the following corollaries of Theorem \ref{thm:whether_codomain_is_over_Fp2}.
For the first corollary, we note that if a twisted Hessian curve $H(a,d)$ satisfies $d = 0$, then we have $j(H(a,0)) = 0$ regardless of $a$, therefore $H(a,0)$ is isomorphic to $H(1,0)$.

\begin{corollary}
\label{cor:when_j_invariant_is_a_cube}
Let $j_0 \in \FF_{p^2}$, and let $H(a,d)$ be a twisted Hessian curve with $a,d \in \FF_{p^2}$ satisfying $j(H(a,d)) = j_0$.
Suppose moreover that $a = 1$ when $d = 0$.
Then we are in precisely one of the following three cases:
\begin{itemize}
\item
$j_0$ is supersingular, $a$ is a cube in $\FF_{p^2}$, and $j_0$ is a cube in $\FF_{p^2}$.
\item
$j_0$ is ordinary and does not lie on the floor $V_h$ of the connected component $V$ of $G_3(\FF_{p^2})$ containing $j_0$, $j_0$ has outdegree $4$ in $V$, $a$ is a cube in $\FF_{p^2}$, and $j_0$ is a cube in $\FF_{p^2}$.
\item
$j_0$ is ordinary and lies on the floor $V_h$ of the connected component $V$ of $G_3(\FF_{p^2})$ containing $j_0$, $j_0$ has outdegree $1$ in $V$, $a$ is not a cube in $\FF_{p^2}$, and $j_0$ is not a cube in $\FF_{p^2}$.
\end{itemize}
\end{corollary}
\begin{proof}
First of all, by Eq.\eqref{eq:j-invariant_Hesse} and as $a,d \in \FF_{p^2}$, $j_0 = j(H(a,d))$ is a cube in $\FF_{p^2}$ if and only if $a$ is a cube in $\FF_{p^2}$.
Hence for each case in the statement, the claim for $j_0$ follows from the claim for $a$.
We also note that the vertex $j_0$ of $G_3(\FF_{p^2})$ always has an edge corresponding to the $3$-isogeny of type $2$ from $H(a,d)$.

First, suppose that $j_0$ is supersingular.
When $d = 0$, we have $a = 1$ by the assumption, which is a cube in $\FF_{p^2}$, as desired.
On the other hand, suppose that $d \neq 0$.
Consider a $3$-isogeny $H(a,d) \to H(A,D)$ of type $1$ associated with $c \in \overline{\FF_p}$ satisfying $c^3 = a$.
Now $H(A,D)$ is also supersingular and hence $j(H(A,D)) \in \FF_{p^2}$.
Therefore $c \in \FF_{p^2}$ by Theorem \ref{thm:whether_codomain_is_over_Fp2}.
This implies that $a = c^3$ is a cube in $\FF_{p^2}$, as desired.
Hence the claim holds in this case.

Secondly, suppose that $j_0$ is ordinary and does not lie on the floor $V_h$ of $V$.
Then $j_0$ has outdegree $4$ in $V$ by the shape of $V$ described in Propositions \ref{prop:volcano_graph} and \ref{prop:volcano_graph_0_1728}.
Now when $d = 0$, we have $a = 1$ by the assumption, which is a cube in $\FF_{p^2}$, as desired.
On the other hand, suppose that $d \neq 0$.
As $j_0$ has outdegree $4$ in $V$ as explained above, for a $3$-isogeny $H(a,d) \to H(A,D)$ of type $1$ associated with $c \in \overline{\FF_p}$ satisfying $c^3 = a$, $j(H(A,D))$ must be in $V$, hence $j(H(A,D)) \in \FF_{p^2}$.
Therefore we have $c \in \FF_{p^2}$ by Theorem \ref{thm:whether_codomain_is_over_Fp2}.
This implies that $a = c^3$ is a cube in $\FF_{p^2}$, as desired.
Hence the claim holds in this case.

Finally, suppose that $j_0$ is ordinary and lies on the floor $V_h$ of $V$.
Assume, for the sake of contradiction, that $a$ is a cube in $\FF_{p^2}$.
Then, as $\FF_{p^2}$ involves a primitive third root of unity, $a$ has three cubic roots in $\FF_{p^2}$, say $c_1,c_2,c_3$.
Then for each $3$-isogeny $H(a,d) \to H(A_j,D_j)$ of type $1$ associated with $c_j$, we have $A_j,D_j \in \FF_{p^2}$ and hence $j(H(A_j,D_j)) \in \FF_{p^2}$.
This implies that $j_0$ has outdegree $4$ in $V$, but this is impossible as now $j_0 \in V_h$ (we note that even if $h = 0$, $j_0$ has outdegree at most $2$ by Propositions \ref{prop:volcano_graph} and \ref{prop:volcano_graph_0_1728}).
Hence $a$ is not a cube in $\FF_{p^2}$.
In particular, we have $d \neq 0$ (if $d = 0$, then $a = 1$ is a cube in $\FF_{p^2}$ by the assumption).
Moreover, if $j_0$ has outdegree at least $2$ in $V$, then there must be a $3$-isogeny $H(a,d) \to H(A,D)$ of type $1$ associated with some $c \in \overline{\FF_p}$ (where $c^3 = a$) satisfying $j(H(A,D)) \in \FF_{p^2}$.
Then we have $c \in \FF_{p^2}$ by Theorem \ref{thm:whether_codomain_is_over_Fp2}, therefore $a = c^3$ is a cube in $\FF_{p^2}$, a contradiction.
Hence $j_0$ has outdegree $1$ in $V$, therefore the claim holds in this case.
This completes the proof.
\end{proof}

For the next two corollaries, we introduce the following terminology.

\begin{definition}
[Twisted Hessian $j$-invariant]
\label{defn:twisted_Hessian_j-invariant}
We say that a $j$-invariant $j_0$ in $\FF_{p^2}$ is \emph{twisted Hessian} if there exists a twisted Hessian curve $H(a,d)$ with $a,d \in \FF_{p^2}$ satisfying $j(H(a,d)) = j_0$.
\end{definition}

We note that any supersingular $j$-invariant in $\FF_{p^2}$ is twisted Hessian by Proposition \ref{prop:from_Weierstrass_to_Hesse}.
Then we have the following corollaries of Theorem \ref{thm:whether_codomain_is_over_Fp2}.

\begin{corollary}
\label{cor:twisted_Hessian_component}
For any ordinary component $V$ of $G_3(\FF_{p^2})$, if $V$ involves a twisted Hessian $j$-invariant, then every vertex of $V$ is twisted Hessian.
\end{corollary}
\begin{proof}
As $V$ is connected by definition, it suffices to show that if two vertices $j_0,j_1 \in V$ are adjacent and $j_0$ is twisted Hessian, then $j_1$ is also twisted Hessian.
By the assumption, there is a twisted Hessian curve $H(a,d)$ with $a,d \in \FF_{p^2}$ satisfying $j(H(a,d)) = j_0$, and there is an elliptic curve $E$ satisfying that $j(E) = j_1$ and $E$ is $3$-isogenous to $H(a,d)$.
Note that when $d = 0$, $a$ can be chosen as $a = 1$.
Now there is a $3$-isogeny $\phi \colon H(a,d) \to H(A,D)$ of type $1$ or type $2$ satisfying $H(A,D) \simeq E$, hence $j(H(A,D)) = j(E) = j_1$.
If $\phi$ is of type $2$, then $A$ and $D$ lie in $\FF_{p^2}$ as well as $a$ and $d$ by definition, therefore $j_1$ is twisted Hessian, as desired.
On the other hand, suppose that $\phi$ is of type $1$ associated with $c \in \overline{\FF_p}$.
If $d = 0$ (hence $a = 1$), then the cubic root $c$ of $a = 1$ lies in $\FF_{p^2}$.
On the other hand, if $d \neq 0$, then the fact $j(H(A,D)) = j_1 \in \FF_{p^2}$ and Theorem \ref{thm:whether_codomain_is_over_Fp2} imply that $c \in \FF_{p^2}$.
Hence we have $c \in \FF_{p^2}$ in any case, therefore $A$ and $D$ lie in $\FF_{p^2}$ by definition and $j_1$ is twisted Hessian, as desired.
This completes the proof.
\end{proof}

\begin{corollary}
\label{cor:no_ordinary_isolated_cycle}
For any ordinary component $V$ of $G_3(\FF_{p^2})$ involving a twisted Hessian $j$-invariant, if the height $h$ of $V$ is $0$, then $V$ has precisely one edge (which is either a self-loop or a non-loop edge).
\end{corollary}
\begin{proof}
By the assumption and Corollary \ref{cor:twisted_Hessian_component}, every vertex of $V$ is twisted Hessian.
As $h = 0$, now $V = V_0$ is a regular graph of degree at most $2$ as in Propositions \ref{prop:volcano_graph} and \ref{prop:volcano_graph_0_1728}.
On the other hand, by Corollary \ref{cor:when_j_invariant_is_a_cube}, each vertex of $V$ has degree $1$ or $4$.
This implies that $V$ is a regular graph of degree $1$, which has precisely one edge, as desired.
This completes the proof.
\end{proof}

As a byproduct, we obtain another proof of the following property of supersingular $j$-invariants originally proved in \cite[Theorem 1.2]{morton2011cubic}.

\begin{corollary}
\label{cor:SS_j_invariant_is_cubic}
Here we temporarily remove the assumption $p \geq 5$.
Any supersingular $j$-invariant in $\FF_{p^2}$ is a cube in $\FF_{p^2}$.
\end{corollary}
\begin{proof}
When $p \geq 5$, any supersingular $j$-invariant in $\FF_{p^2}$ is twisted Hessian as mentioned above, therefore the claim is a part of Corollary \ref{cor:when_j_invariant_is_a_cube}.
On the other hand, when $p \in \{2,3\}$, $0$ is the only supersingular $j$-invariant, therefore the claim holds obviously.
\end{proof}

\subsection{The Case over $\FF_p$ with $p \equiv 2 \pmod{3}$}
\label{sec:3-isogeny_walk:2_mod_3}

From now, we consider the case $p \equiv 2 \pmod{3}$ and focus on twisted Hessian curves $H(a,d)$ with $a,d \in \FF_p$.
For this $p$, the map $\FF_p \to \FF_p$, $x \mapsto x^3$ is a bijection.
Therefore each element $x$ of $\FF_p$ has a unique cubic root in $\FF_p$, denoted here by $\sqrt[3]{x}$.
Now $H(a,d)$ is isomorphic to a Hessian curve $H(1,d/\sqrt[3]{a})$ via a change of coordinates $(X',Y',Z') = (\sqrt[3]{a}X,Y,Z)$; based on this fact, we focus on Hessian curves $H(1,d)$ with $d \in \FF_p$ from now.

For a Hessian curve $H(1,d)$ with $d \in \FF_p$, we consider a $3$-isogeny $\phi \colon H(1,d) \to H(A,D)$ of type $1$ associated with $c := 1 = \sqrt[3]{1}$, where $A = d^2 + 3d + 9$ and $D = d + 6$.
Then we call the composition $\phi' \colon H(1,d) \to H(A,D) \to H(1,D')$ of $\phi$ followed by the change of coordinates $(X',Y',Z') = (\sqrt[3]{A}X,Y,Z)$ the $3$-isogeny of \emph{type $1'$}.
Then we have $D' = (d + 6) / \sqrt[3]{d^2 + 3d + 9}$.
Moreover, as the latter change of coordinates maps the point $(0:-\omega:1)$ of $H(A,D)$ to the point $(0:-\omega:1)$ of $H(1,D')$ where $\omega$ is a primitive third root of unity, the same proof as Proposition \ref{prop:backtracking} implies the following property.

\begin{lemma}
\label{lem:backtracking_type1'}
Assume that $p \equiv 2 \pmod{3}$.
Let $\phi \colon H(1,d) \to H(1,D)$ be the $3$-isogeny of type $1'$ as above from a Hessian curve $H(1,d)$ with $d \in \FF_p$.
Then the $3$-isogeny $\psi \colon H(1,D) \to H(a',d')$ of type $2$ from the Hessian curve $H(1,D)$ (as in Proposition \ref{prop:3-isogeny_Hesse_type2}) satisfies $H(a',d') \simeq H(1,d)$.
\end{lemma}

We define $\mathcal{P} := \{ d \in \FF_p \mid 27 - d^3 \neq 0 \}$.
The next lemma shows that the correspondence $d \mapsto D$ of parameters for Hessian curves over $\FF_p$ induced by the $3$-isogeny of type $1'$ is injective.

\begin{lemma}
\label{lem:type1'_parameter_injective}
Assume that $p \equiv 2 \pmod{3}$.
Let $d_1,d_2 \in \mathcal{P}$.
If $D = (d_1 + 6) / \sqrt[3]{d_1{}^2 + 3d_1 + 9} = (d_2 + 6) / \sqrt[3]{d_2{}^2 + 3d_2 + 9}$, then $D \in \mathcal{P}$ and $d_1 = d_2$.
\end{lemma}
\begin{proof}
First, for each $j \in \{1,2\}$, we have
\[
27 - D^3
= 27 - \frac{ d_j{}^3 + 18 d_j{}^2 + 108 d_j + 216 }{ d_j{}^2 + 3 d_j + 9 }
= \frac{ - d_j{}^3 + 9 d_j{}^2 - 27 d_j + 27 }{ d_j{}^2 + 3 d_j + 9 }
= \frac{ (3 - d_j)^3 }{ d_j{}^2 + 3 d_j + 9 } \enspace,
\]
therefore $27 - D^3 \neq 0$ (note that $d_j \neq 3$ as $d_j \in \mathcal{P}$), hence $D \in \mathcal{P}$.
Moreover, we have $\sqrt[3]{27 - D^3} = (3 - d_j) / \sqrt[3]{ d_j{}^2 + 3 d_j + 9 }$, therefore
\[
\frac{ 3 ( D - 2 \cdot \sqrt[3]{27 - D^3} ) }{ D + \sqrt[3]{27 - D^3} }
= \frac{ 3 ( (d_j + 6) - 2 (3 - d_j) ) / \sqrt[3]{ d_j{}^2 + 3 d_j + 9 } }{ ( (d_j + 6) + (3 - d_j) ) / \sqrt[3]{ d_j{}^2 + 3 d_j + 9 }}
= \frac{ 3 \cdot 3 d_j }{ 9 }
= d_j \enspace.
\]
As this holds for both $j \in \{1,2\}$, we have $d_1 = d_2$, as desired.
This completes the proof.
\end{proof}

Now we have the following result for the positions of Hessian curves defined over $\FF_p$ and $3$-isogenies from such curves in the $3$-isogeny graph.

\begin{theorem}
\label{thm:surface_2_mod_3}
Assume that $p \equiv 2 \pmod{3}$.
Let $H(1,d)$ be an ordinary Hessian curve with $d \in \FF_p$, and let $j_0 := j(H(1,d))$.
Let $V$ be the connected component of the $3$-isogeny graph $G_3(\FF_{p^2})$ containing $j_0$, as in Proposition \ref{prop:volcano_graph} or Proposition \ref{prop:volcano_graph_0_1728}.
Then:
\begin{enumerate}
\item \label{thm:surface_2_mod_3:item:1}
$j_0$ lies on the surface $V_0$ of $V$.
\item \label{thm:surface_2_mod_3:item:2}
Both of the $3$-isogeny of type $1'$ from $H(1,d)$ and the $3$-isogeny of type $2$ from $H(1,d)$ lie on the surface $V_0$ of $V$.
\end{enumerate}
\end{theorem}
\begin{proof}
Put $d_0 := d \in \mathcal{P}$, and consider a sequence $d_0 \to d_1 \to d_2 \to \cdots$ of elements $d_k$ of $\mathcal{P}$ given by $d_{k+1} = (d_k + 6) / \sqrt[3]{d_k{}^2 + 3 d_k + 9}$ (note that each $d_k$ is in fact an element of $\mathcal{P}$ by Lemma \ref{lem:type1'_parameter_injective}).
Now for each index $i \geq 0$, there is the $3$-isogeny $\phi_i$ of type $1'$ from $H(1,d_i)$ to $H(1,d_{i+1})$.
Hence each $j(H(1,d_i))$ belongs to the same ordinary component $V$ of $G_3(\FF_{p^2})$ as $j_0 = j(H(1,d))$. 
On the other hand, as $\mathcal{P}$ is a finite set, there is an index $k \geq 1$ satisfying that $d_k = d_{k'}$ for some $0 \leq k' \leq k - 1$.
Take the minimum value of such a $k$.
Now if $k' \geq 1$, then we have $d_{k-1} = d_{k'-1}$ by Lemma \ref{lem:type1'_parameter_injective}, contradicting the choice of $k$.
Hence we have $k' = 0$, i.e., $d_k = d_0$.

To prove that $j(H(1,d_i))$ lies on the surface $V_0$ of $V$ for every $i \geq 0$, assume, for the sake of contradiction, that this is not the case.
Then we can take an index $\nu \geq 1$ satisfying that $j(H(1,d_{\nu}))$ is not on the surface of $V$ and is one of the vertices closest to the floor of $V$ among the vertices $j(H(1,d_i))$ with $i \geq 0$ (recall that $d_k = d_0$, therefore such an index $\nu$ can be chosen from the range $\nu \geq 1$).
Now by the shape of $V$ as shown in Proposition \ref{prop:volcano_graph} and Proposition \ref{prop:volcano_graph_0_1728}, $\phi_{\nu-1} \colon H(1,d_{\nu-1}) \to H(1,d_{\nu})$ is a descending isogeny and $\phi_{\nu} \colon H(1,d_{\nu}) \to H(1,d_{\nu+1})$ is an ascending isogeny.
Therefore by Lemma \ref{lem:backtracking_type1'} applied to $\phi_{\nu-1}$, the $3$-isogeny $\psi$ of type $2$ from $H(1,d_{\nu})$ is also an ascending isogeny with codomain curve isomorphic to $H(1,d_{\nu-1})$.
Now $\phi_{\nu}$ and $\psi$ are different ascending isogenies from the same $H(1,d_{\nu})$, but this is impossible due to the shape of the graph $V$ as in Proposition \ref{prop:volcano_graph} and Proposition \ref{prop:volcano_graph_0_1728}.
Hence $j(H(1,d_i))$ lies on the surface $V_0$ for every $i \geq 0$.
In particular, $j(H(1,d)) = j(H(1,d_0)) = j_0$ lies on the surface $V_0$.
This proves Claim \ref{thm:surface_2_mod_3:item:1}.

Recall that $k \geq 1$ and $d_k = d_0 = d$.
Now the codomain curve $H(1,d_{k+1})$ of the $3$-isogeny $\phi_k$ of type $1'$ from $H(1,d_k) = H(1,d)$ lies on the surface $V_0$ as shown above.
On the other hand, Lemma \ref{lem:backtracking_type1'} applied to the $3$-isogeny $\phi_{k-1}$ of type $1'$ shows that the codomain curve of the $3$-isogeny of type $2$ from $H(1,d_k) = H(1,d)$ is isomorphic to $H(1,d_{k-1})$, which lies on the surface $V_0$ as shown above.
Hence, both of the $3$-isogeny of type $1'$ from $H(1,d)$ and the $3$-isogeny of type $2$ from $H(1,d)$ lie on the surface $V_0$.
This proves Claim \ref{thm:surface_2_mod_3:item:2}.
This completes the proof of Theorem \ref{thm:surface_2_mod_3}.
\end{proof}

By Theorem \ref{thm:surface_2_mod_3} and the shape of the ordinary components of $G_3(\FF_{p^2})$ as in Proposition \ref{prop:volcano_graph} and Proposition \ref{prop:volcano_graph_0_1728}, we have the following corollary.

\begin{corollary}
\label{cor:surface_2_mod_3}
Assume that $p \equiv 2 \pmod{3}$.
Let $H(1,d)$ with $d \in \FF_p$ be an ordinary Hessian curve.
Let $V$ be the connected component of the $3$-isogeny graph $G_3(\FF_{p^2})$ containing $j(H(1,d))$, as in Proposition \ref{prop:volcano_graph} or Proposition \ref{prop:volcano_graph_0_1728}.
Then $j(H(1,d))$ lies on the surface $V_0$ of $V$, and any $3$-isogeny of type $1$ from $H(1,d)$ associated with a primitive third root $c$ of unity (i.e., $c^3 = 1$ and $c \neq 1$) is a descending isogeny.
\end{corollary}

\section{Supersingularity Testing for Hessian Curves}
\label{sec:SStest}

In \cite{DBLP:journals/lmsjcm/Sutherland12a}, based on the $2$-volcano structure of ordinary components of $2$-isogeny graphs $G_2(\FF_{p^2})$, Sutherland proposed a deterministic algorithm (given a quadratic non-residue and a cubic non-residue in $\FF_{p^2}$ as auxiliary inputs) with time complexity $O(n^3 \log^2 n)$, where $n = \log p$, to determine whether a given elliptic curve $E/\FF_{p^2}$ is supersingular.
With a similar idea (i.e., any three $3$-isogenies from a given ordinary elliptic curve involve at least one descending isogeny), based on the properties of $3$-isogenies from twisted Hessian curves described above (and also on Proposition \ref{prop:bound_for_height} for the bound for the heights), we obtain a deterministic supersingularity testing algorithm (given a cubic non-residue in $\FF_{p^2}$ as an auxiliary input) for twisted Hessian curves $H(a,d)$ with $a,d \in \FF_{p^2}$ as follows.
Now we recall that we have assumed that $p \geq 5$, and recall from Section \ref{sec:preliminaries:EC} that $0 \in \FF_{p^2}$ is ordinary when $p \equiv 1 \pmod{3}$ and is supersingular when $p \equiv 2 \pmod{3}$.

\paragraph{Algorithm 1: Supersingularity testing for twisted Hessian curves $H(a,d)$ with $a,d \in \FF_{p^2}$}

\begin{enumerate}
\item
When $d = 0$, output $\texttt{true}$ if $p \equiv 2 \pmod{3}$ and output $\texttt{false}$ otherwise.
\item
Set $h_0 \leftarrow \lfloor \log_3(4p) / 2 \rfloor$.
\item
If $a$ is not a cube in $\FF_{p^2}$, then output \texttt{false}.
Otherwise, compute a cubic root $c \in \FF_{p^2}$ of $a$; and for each $j \in \{0,1,2\}$, set $c_j \leftarrow c \omega^j$ where $\omega \in \FF_{p^2}$ is a primitive third root of unity, $a_j \leftarrow d^2 c_j + 3 d c_j{}^2 + 9 a$, and $d_j \leftarrow d + 6 c_j$.
Moreover, set $h \leftarrow 1$.
\item
While $h \leq h_0$, repeat the following:
\begin{enumerate}
\item
For each $j \in \{0,1,2\}$, do the following:
\begin{enumerate}
\item
When $d_j = 0$, output $\texttt{true}$ if $p \equiv 2 \pmod{3}$ and output $\texttt{false}$ otherwise.
\item
Set $a' \leftarrow a_j$ and $d' \leftarrow d_j$.
\item
If $a'$ is not a cube in $\FF_{p^2}$, then output \texttt{false}.
Otherwise, if $h < h_0$, then compute a cubic root $c \in \FF_{p^2}$ of $a'$, and set $a_j \leftarrow d'{}^2 c + 3 d' c^2 + 9 a'$ and $d_j \leftarrow d' + 6 c$.
\end{enumerate}
\item
Set $h \leftarrow h + 1$.
\end{enumerate}
\item
Output \texttt{true}.
\end{enumerate}
However, our computer experiment showed that Algorithm 1 for a general twisted Hessian curve $H(a,d)$ with $a,d \in \FF_{p^2}$ does not outperform an improved version of Sutherland's algorithm using $2^2$-isogenies for Legendre curves given in \cite{DBLP:conf/casc/HashimotoN21,DBLP:journals/ieiceta/HashimotoN23}.

On the other hand, when $p \equiv 2 \pmod{3}$ and the input curve $H(a,d)$ is defined over $\FF_p$ (i.e., $a,d \in \FF_p$) and is ordinary, a descending $3$-isogeny from $H(a,d) \simeq H(1,d / \sqrt[3]{a})$ is specified as in Corollary \ref{cor:surface_2_mod_3}.
By using this fact, the algorithm above can be improved for this special case as follows (note that a cubic root of $a \in \FF_p$ in this case can be computed as $a^{(2p-1)/3}$).

\paragraph{Algorithm 2: Supersingularity testing for twisted Hessian curves $H(a,d)$ with $a,d \in \FF_p$, when $p \equiv 2 \pmod{3}$}

\begin{enumerate}
\item
When $d = 0$, output \texttt{true}.
\item
Set $h_0 \leftarrow \lfloor \log_3(4p) / 2 \rfloor$.
\item
Compute $c \leftarrow a^{(2p-1)/3}$; set $d' \leftarrow d/c$; and set $a \leftarrow d'{}^2 \omega + 3 d' \omega^2 + 9$ and $d \leftarrow d' + 6 \omega$.
Moreover, set $h \leftarrow 1$.
\item
While $h \leq h_0$, repeat the following:
\begin{enumerate}
\item
When $d = 0$, output \texttt{true}.
\item
Set $a' \leftarrow a$ and $d' \leftarrow d$.
\item
If $a'$ is not a cube in $\FF_{p^2}$, then output \texttt{false}.
Otherwise, if $h < h_0$, then compute a cubic root $c \in \FF_{p^2}$ of $a'$, and set $a \leftarrow d'{}^2 c + 3 d' c^2 + 9 a'$ and $d \leftarrow d' + 6 c$.
\item
Set $h \leftarrow h + 1$.
\end{enumerate}
\item
Output \texttt{true}.
\end{enumerate}
We note that when $p \equiv 2 \pmod{3}$ and an input is given as a form of $j$-invariant $j_0 \in \FF_p$ (or a short Weierstrass form over $\FF_p$), we can convert the input to a twisted Hessian curve $H(a,d)$ as in Proposition \ref{prop:from_Weierstrass_to_Hesse} at the beginning (if this conversion failed, then we can conclude that the input is ordinary), and then apply Algorithm 2 to $H(a,d)$.

For Algorithm 2, we performed computer experiments to confirm the correctness of the algorithm.
Specifically, for each bit length of $p$ as in Table \ref{tab:experiment_Fp_correctness}, we chose $10000$ random primes $p$ with $p \equiv 2 \pmod{3}$ and for each $p$, we chose random parameters $a,d \in \FF_p$ with $a (27a - d^3) \neq 0$.
Then we execute our proposed algorithm (Algorithm 2) for the twisted Hessian curve $H(a,d)$.
On the other hand, we compute the $j$-invariant $j_0$ of $H(a,d)$, compute a short Weierstrass form of an elliptic curve $E$ with $j(E) = j_0$, and check whether $E$ is supersingular by using the command \texttt{E.is\_supersingular()} of computer algebra system SageMath.
The result is summarized in Table \ref{tab:experiment_Fp_correctness}, which shows that our proposed algorithm always correctly decided whether the given curve is supersingular.

\begin{table}[t]
\centering
\caption{Experimental results for the correctness of Algorithm 2 (with $p \equiv 2 \pmod{3}$ and $a,d \in \FF_p$).}
\label{tab:experiment_Fp_correctness}
\begin{tabular}{c|c|c|c|c}
Length of $p$ & \multicolumn{2}{c|}{Ordinary} & \multicolumn{2}{c}{Supersingular} \\ \cline{2-5}
(bit) & \# of curves & \# of failure & \# of curves & \# of failure \\ \hline
$10$ & $9509$ & $0$ & $491$ & $0$ \\ \hline
$15$ & $9892$ & $0$ & $108$ & $0$ \\ \hline
$20$ & $9984$ & $0$ & $16$ & $0$ \\ \hline
$25$ & $9997$ & $0$ & $3$ & $0$ \\ \hline
$30$ & $10000$ & $0$ & $0$ & $0$
\end{tabular}
\end{table}

Moreover, Table \ref{tab:experiment_Fp_SS} shows a comparison of efficiency of our proposed algorithm (Algorithm 2) with the aforementioned existing algorithm using Legendre curves in \cite{DBLP:journals/ieiceta/HashimotoN23} (combined with an improved bound for the height given in \cite{hashimoto2024boundsheights2isogenygraphs}) by computer experiments.
In the experiment, we implemented both algorithms on SageMath, where we used Tonelli--Shanks-like algorithms for computing cubic roots in our algorithm and quartic roots in the existing algorithm.
The CPU and memory environments for the experiment were 13th Gen Intel Core i5-1335U (1.30 GHz) with 16.0 GB RAM.
In the experiment, we only used supersingular input curves in order to measure the worst-case computing times for the algorithms.
Specifically, for each bit length of $p$ as in Table \ref{tab:experiment_Fp_SS}, we chose $10$ random primes $p$ with $p \equiv 2 \pmod{3}$ (and generated a cubic non-residue in $\FF_{p^2}$ and a quartic non-residue in $\FF_{p^2}$ for each $p$).
For each $p$, we chose $N$ supersingular twisted Hessian curves $H(a,d)$ with $a,d \in \FF_p$ where $N = 100$ for bit lengths $\leq 128$ of $p$ and $N = 10$ for bit lengths $\geq 256$ of $p$.
Here each $H(a,d)$ was generated by applying $100$ random $3$-isogenies defined over $\FF_p$, starting from a supersingular Hessian curve $H(1,0) \colon X^3 + Y^3 + Z^3 = 0$ with $j$-invariant $0$.
Moreover, for each $H(a,d)$, we computed a short Weierstrass form of an elliptic curve $E$ with the same $j$-invariant as $H(a,d)$, and converted it into Legendre form.
Then we executed our proposed algorithm for $H(a,d)$ and the existing algorithm for the Legendre form of $E$, and measured their computing times (in particular, we excluded the computing times for generating cubic/quartic non-residues from the measured computing times).
Then Table \ref{tab:experiment_Fp_SS} shows average computing times (in seconds).
As shown in the table, for these parameters, our proposed algorithm significantly improves the efficiency in comparison to the existing algorithm in \cite{DBLP:journals/ieiceta/HashimotoN23}.
We also note that our proposed algorithm always output the correct answer \texttt{true} in all the trials.

\begin{table}[t]
\centering
\caption{Comparison of average computing times for Algorithm 2 (with $p \equiv 2 \pmod{3}$ and $a,d \in \FF_p$) with the existing algorithm in \cite{DBLP:journals/ieiceta/HashimotoN23}.}
\label{tab:experiment_Fp_SS}
\begin{tabular}{c|c|c|c}
Length of $p$ (bit) & Our Algorithm 2 (sec.) & \cite{DBLP:journals/ieiceta/HashimotoN23} (sec.) & Ratio (Ours / \cite{DBLP:journals/ieiceta/HashimotoN23}) \\ \hline
$16$ & $3.42 \times 10^{-4}$ & $8.80 \times 10^{-4}$ & $38.9\%$ \\ \hline
$32$ & $5.57 \times 10^{-4}$ & $1.52 \times 10^{-3}$ & $36.6\%$ \\ \hline
$64$ & $1.39 \times 10^{-3}$ & $3.78 \times 10^{-3}$ & $36.8\%$ \\ \hline
$128$ & $1.72 \times 10^{-2}$ & $4.27 \times 10^{-2}$ & $40.3\%$ \\ \hline
$256$ & $8.11 \times 10^{-2}$ & $1.90 \times 10^{-1}$ & $42.7\%$ \\ \hline
$384$ & $2.32 \times 10^{-1}$ & $5.19 \times 10^{-1}$ & $44.7\%$ \\ \hline
$512$ & $5.47 \times 10^{-1}$ & $1.23 \times 10^{0\phantom{-}}$ & $44.5\%$ \\ \hline
$640$ & $9.21 \times 10^{-1}$ & $2.16 \times 10^{0\phantom{-}}$ & $42.6\%$ \\ \hline
$768$ & $1.65 \times 10^{0\phantom{-}}$ & $3.76 \times 10^{0\phantom{-}}$ & $43.9\%$ \\ \hline
$896$ & $2.66 \times 10^{0\phantom{-}}$ & $5.98 \times 10^{0\phantom{-}}$ & $44.5\%$ \\ \hline
$1024$ & $4.24 \times 10^{0\phantom{-}}$ & $9.14 \times 10^{0\phantom{-}}$ & $46.4\%$
\end{tabular}
\end{table}

\paragraph{Acknowledgements.}

This work was supported by JSPS KAKENHI Grant Numbers JP24K17281 and JP25K14994, Japan.

\bibliographystyle{plain}
\bibliography{HesseSStest}

\end{document}